\documentclass[amssymb, 11pt]{amsart}
\usepackage{latexsym, amssymb, appendix}

\newcommand{\F}{\mathbb{F}}
\newcommand{\Irr}{\mathrm{Irr}}

\newcommand{\vep}{\varepsilon}
\newcommand{\epi}{\varepsilon_{\iota}}
\newcommand{\Om}{\Omega}
\newcommand{\POm}{\mathrm{P}\Omega}
\newcommand{\PSU}{\mathrm{PSU}}
\newcommand{\gO}{\mathrm{O}}
\newcommand{\mmod}{\mathrm{mod} \;}
\newcommand{\SO}{\mathrm{SO}}
\newcommand{\K}{\mathrm{K}}

\newcommand{\gT}{\mathrm{T}}
\newcommand{\Sp}{\mathrm{Sp}}

\def\adots{\mathinner{\mkern2mu\raise0pt\hbox{.}  
\mkern2mu\raise4pt\hbox{.}\mkern1mu
\raise7pt\vbox{\kern7pt\hbox{.}}\mkern1mu}}

\newtheorem{theorem}{Theorem}[section]
\newtheorem{proposition}{Proposition}[section]
\newtheorem{lemma}{Lemma}[section]

\numberwithin{equation}{section}

\allowdisplaybreaks

\begin{document}

\title[Reality properties of finite simple orthogonal groups]{Some reality properties of finite simple orthogonal groups}

\author{Jiwon Kim}
\address{107, Haedoji-ro, 6-4101\\Yeonsu-gu, Incheon, 21977\\Republic of Korea}
\email{kjiwon831@gmail.com}

\author{Stephen Trefethen}
\address{Department of Mathematics, William \& Mary\\ Williamsburg, VA, 23187\\USA}
\email{sjtrefethen@wm.edu, vinroot@math.wm.edu}

\author{C. Ryan Vinroot}

\begin{abstract}  
We prove several reality properties for finite simple orthogonal groups.  For any prime power $q$ and $m \geq 1$, we show that all real conjugacy classes are strongly real in the simple groups $\POm^{\pm}(4m+2,q), m \geq 1$, except in the case $\POm^{-}(4m+2,q)$ with $q \equiv 3(\mmod 4)$, and we construct weakly real classes in this exceptional case for any $m$.  We also show that no irreducible complex character of $\POm^{\pm}(n,q)$ can have Frobenius--Schur indicator $-1$, except possibly in the case $\POm^{-}(4m+2,q)$ with $q \equiv 3(\mmod 4)$.
\end{abstract}

\maketitle

\thispagestyle{empty}

\section{Introduction}

Let $\gO(V)$ be the full orthogonal group corresponding to a non-degenerate symmetric form on a finite-dimensional vector space $V$ over a field of characteristic which is not 2.  It is a result of M. Wonenburger \cite{Wo66} that every element of $\gO(V)$ is the product of two involutions from $\gO(V)$.  This is equivalent to the fact that every element of $\gO(V)$ may be conjugated to its inverse by an involution in $\gO(V)$.  That is, every conjugacy class of the orthogonal group $\gO(V)$ is \emph{strongly real}.  For certain subgroups of orthogonal groups, this property is further explored by F. Kn\"{u}ppel and G. Thomsen \cite{KnTh98}, who classify when the commutator subgroup $\Om(V)$ of $\gO(V)$, and the kernel of the spinor norm on $\gO(V)$, are strongly real.  

If $\gO(V)$ is now the orthogonal group corresponding to a non-defective quadratic form on $V$ over a field which is characteristic 2, it was shown by E. Ellers and W. Nolte \cite{ElNo82} and R. Gow \cite{Gow81} that $\gO(V)$ is also strongly real.  In the case that $V$ has dimension divisible by $4$ and $V$ is defined over a finite field of characteristic $2$, it is proven by J. R\"{a}m\"{o} \cite{Ra11} that the commutator subgroup $\Om(V)$ of $\gO(V)$ is strongly real.  This is not the case if the dimension of $V$ is 2 modulo $4$, by the work of P. H. Tiep and A. E. Zalesski \cite{TiZa05}. 

For an element (or conjugacy class) of a group to be strongly real, it is necessary that it is \emph{real}, or conjugate to its inverse within the group.  In the representation theory of finite groups, it is of interest to understand which conjugacy classes of a group are real, because the number of real conjugacy classes of a finite group is equal to the number of real-valued irreducible complex characters of that group.  The classification of finite simple groups with the property that all elements are real is completed by P. H. Tiep and A. E. Zalesski \cite{TiZa05}.  The classification of finite simple groups such that all elements are strongly real is completed by J. R\"{a}m\"{o} \cite{Ra11}, and E. P. Vdovin and A. A. Gal$^\prime$t \cite{VdGa10}, with the somewhat surprising result that all elements of a finite simple group are real if and only if they are all strongly real.

In this paper, we focus on the finite simple orthogonal groups, and the question of when all real elements are strongly real.  By the results mentioned above, the answer to this question is understood in the case when all elements are real, and so we study the remaining cases.  We give a complete answer to the question in the case that the underlying vector space has dimension which is 2 modulo 4.

The connection between the strongly real classes of a finite group and the irreducible complex characters of that group is still mysterious in the general case, and one motivation for our study is to further understand this connection in the case of finite simple orthogonal groups.  In particular, in some classes of finite groups there is a connection between strongly real classes and the irreducible complex characters which are afforded by a real representation.  For example, it is proven in \cite{Vi202} that for a finite simple group, all conjugacy classes are strongly real if and only if all irreducible complex characters can be afforded by a real representation.  In this paper we gather evidence that for finite simple orthogonal groups, all real classes are strongly real if and only if all real-valued irreducible complex characters can be afforded by real representations.

This paper is organized as follows.  In Section \ref{basic}, we give some standard information on orthogonal groups over finite fields and their important subgroups.  In Section \ref{OddChar}, after giving some previous results, we prove our main results in the case of odd characteristic, which are reality results for $\Om(V)$ and $\POm(V)$ in the remaining cases when $\dim(V) \equiv 2($mod $4)$.  In particular, in the cases when these groups are known to have classes which are not real, we show that all real elements are strongly real, in Proposition \ref{easycases} and Theorem \ref{realomega}, except in the case of $\Om^-(4m+2,q)$ and $\POm^-(4m+2,q)$ with $q \equiv 3($mod $4)$.  In that exceptional case, we construct elements which are real but not strongly real in Theorem \ref{Weakly} (with technical base case constructions given in the Appendix).  In Section \ref{CharTwo}, we prove that all real elements are strongly real in $\Om^{\pm}(4m+2, q)$ when $q$ is a power of $2$, and we classify all real elements, in Theorem \ref{OmStrong}.  In Section \ref{Indicators}, we prove in Theorem \ref{RealVal} that all real-valued irreducible characters of $\Om^{\pm}(n,q)$ and $\POm^{\pm}(n,q)$ can be afforded by real representations, except possibly in the case $\Om^-(4m+2,q)$ and $\POm^-(4m+2,q)$ with $q \equiv 3($mod $4)$, and we give some concluding remarks.
\\
\\
{\bf Acknowledgments. }  Much of this work was done when the first-named author was a visiting assistant professor at William \& Mary.  The third-named author was supported in part by a grant from the Simons Foundation, Award \#713090.

\section{Finite orthogonal groups} \label{basic}

In this section we state some well-known facts about finite orthogonal groups and some of their subgroups.  All of these results may be found in \cite{Gr02} or \cite[Chapter 16]{CaEn04}.

Let $q$ be the power of a prime, and $\F_q$ a finite field with $q$ elements.  Fix a finite-dimensional $\F_q$-vector space $V$, say $\dim_{\F_q}(V) = n$, and we assume $V$ carries a quadratic form $Q$.  In the case that $q$ is odd, we assume that the bilinear form $B$ associated with $Q$ is non-degenerate, and in the case that $q$ is even we assume that $n$ is even and $Q$ is non-defective.  That is, when $q$ is odd we have, for all $u, v \in V$,
$$Q(u+v) = Q(u) + 2B(u,v) + Q(v),$$
for a non-degenerate symmetric bilinear form $B$, while if $q$ is even we have
$$Q(u+v) = Q(u) + B(u,v) + Q(v),$$
for a non-degenerate symplectic form $B$.  We may then consider the group $\gO(V)$ of isometries of the quadratic form $Q$, which is the {\it orthogonal group}.  When $q$ is odd, this is the same as the group of isometries of the associated symmetric bilinear form $B$.

For all $n$ and $q$ considered above, there are precisely two equivalence classes of quadratic forms $Q$.  For $v \in V$, we take $v$ to have coordinate representation $v = (x_1, \ldots, x_n)$ with respect to some basis.  When $n=2m$ is even, one form may be taken to be
$$Q(v) = x_1 x_2 + x_3 x_4 + \cdots + x_{2m-1} x_{2m},$$
which we say is a {\it split} form, or {\it $+$-type}.  When $q$ is odd and $n = 2m+1$ is odd, the split, or $+$-type, form is given by 
$$Q(v) = x_1 x_2 + x_3 x_4 + \cdots + x_{2m-1} x_{2m} + x_{2m+1}^2.$$
The second equivalence class of quadratic forms (which we do not formulate explicitly here) in the cases under consideration will be called {\it non-split} forms, or {\it $-$-type}.  We will denote by $\gO^{+}(n, q)$ or $\gO^{-}(n,q)$ the orthogonal groups corresponding to a quadratic form of split or non-split type, respectively, and we will write $\gO^{\pm}(n,q)$ for the generic case.  In the case that $n$ and $q$ are odd, we in fact have $\gO^+(n,q) \cong \gO^-(n,q)$, but these groups are not isomorphic in the general case.

One group of primary interest in this paper is the derived (or commutator) subgroup of the orthogonal group.  Letting $X^{\prime}$ denote the derived subgroup of a group $X$, we may define
$$ \Om^{\pm}(n,q) = \gO^{\pm}(n,q)^{\prime}.$$
When $q$ is even, $\Om^{\pm}(n,q)$ is an index two subgroup of $\gO^{\pm}(n,q)$ and is a simple group.  When $q$ is odd, $\Om^{\pm}(n,q)$ is an index $4$ subgroup of $\gO^{\pm}(n,q)$, and may or may not contain the central involution of $\gO^{\pm}(n,q)$.  Whether the central involution is contained in $\Om^{\pm}(n,q)$ may be determined by $n$, $q$, and the type $\pm$ of the underlying form, as described in \cite[pg. 230]{CaEn04}.  If $Z$ is the center of $\Om^{\pm}(n,q)$, then $\POm^{\pm}(n,q) = \Om^{\pm}(n,q)/Z$ is a finite simple group (with a finite number of exceptions when $n$ is small).

We now assume that $q$ is the power of an odd prime until Section \ref{CharTwo}.  In this case, we may define the special orthogonal groups $\SO^{\pm}(n,q)$ to be the elements of $\gO^{\pm}(n,q)$ with determinant $1$.  Then $\Om^{\pm}(n,q)$ is an index $2$ subgroup of $\SO^{\pm}(n,q)$, and $\SO^{\pm}(n,q)$ is an index $2$ subgroup of $\gO^{\pm}(n,q)$.  One may also define the {\it spinor norm} on $\gO^{\pm}(n,q)$, which is a multiplicative homomorphism
$$ \theta: \gO^{\pm}(n,q) \longrightarrow \F_q^{\times}/(\F_q^{\times})^2,$$
and the explicit definition can be found in \cite[Chapter 9]{Gr02}.  Then $\ker(\theta)$ is an index $2$ subgroup of $\gO^{\pm}(n,q)$, and we write $\ker(\theta) = \K^{\pm}(n,q)$.  Also $\Om^{\pm}(n,q)$ is an index 2 subgroup of $\K^{\pm}(n,q)$, and in fact we have
$$ \Om^{\pm}(n,q) = \SO^{\pm}(n,q) \cap \K^{\pm}(n,q).$$
There is a third index two subgroup of $\gO^{\pm}(n,q)$ containing $\Om^{\pm}(n,q)$, consisting of $\Om^{\pm}(n,q)$ together with the coset in $\gO^{\pm}(n,q)$ disjoint from both $\K^{\pm}(n,q)$ and $\SO^{\pm}(n,q)$.  We define this subgroup to be $\gT^{\pm}(n,q)$, that is,
\begin{align*}
\gT^{\pm}(n,q) & = \Om^{\pm}(n,q) \cup \{ g \in \gO^{\pm}(n,q) \, \mid \, g \not\in \K^{\pm}(n,q) \cup \SO^{\pm}(n,q) \} \\
                        & = \mathrm{ker} (\theta \times \det). 
\end{align*}

We denote the \emph{discriminant} of the symmetric form $B$ carried by $V$, defined in \cite[Chapter 2]{Gr02}, by $dV \in \F_q^{\times}/(\F_q^{\times})^2$.  

The discriminant, spinor norm, and determinant each behave as a direct product map over an orthogonal direct sum of the underlying space $V = \F_q^n$.  That is, suppose $V = \oplus_i V_i$ is an orthogonal direct sum with respect to $B$, so that we may consider the non-degenerate symmetric form $B$ restricted to the subspace $V_i$, denoted $B_{V_i}$.  Then we have $dV = \sum_i dV_i$.  If $g \in \gO(V)$ and each $V_i$ is $g$-invariant, then we may consider $g$ with restricted action on each subspace $V_i$, and denote this by $g_i = g_{V_i} \in \gO(V_i)$.  Then we have that the spinor norm and determinant maps satisfy
$$ \theta(g) = \sum_i \theta(g_i) \quad \text{ and } \quad \det(g) = \prod_i \det(g_i).$$
When restricting the symmetric form $B$ to a subspace $V_i$, the $\pm$ type of the corresponding quadratic form may be determined by the values of $q$, $\dim(V_i)$, and the discriminant $dV_i$, as described in \cite[pgs. 221-222]{CaEn04}.

\section{Odd characteristic} \label{OddChar}

We take $q$ to be the power of an odd prime, $V = \F_q^n$ endowed with a quadratic form $Q$, and $B$ the associated non-degenerate symmetric bilinear form on $V$.  Let $\gO^{\pm}(n, q)$ be the corresponding orthogonal group.

\subsection{Preliminary Results}

In this section we give previous results which will be repeatedly used in what follows.  If $G$ is a finite group, then we say $g \in G$ is \emph{real in $G$} if there is an element $h \in G$ such that $h g h^{-1} = g^{-1}$.  If all elements of $G$ satisfy this property, then we say $G$ is a \emph{real group}.  We say $g \in G$ is \emph{strongly real in $G$} if there is an element $s \in G$ such that $s^2 = 1$ and $s g s^{-1} = g^{-1}$, or equivalently, if there are elements $s_1, s_2 \in G$ such that $s_1^2 = s_2^2 = 1$ and $g = s_1 s_2$.  If all elements of $G$ are strongly real in $G$, then we say $G$ is a \emph{strongly real group}.  

We have the following result of Wonenburger \cite{Wo66} mentioned in the introduction (see also \cite[Proposition 4.5]{KnTh98}).

\begin{theorem} \label{Osr}  If $q$ is the power of an odd prime and $n \geq 2$, then $\gO^{\pm}(n,q)$ is strongly real.
\end{theorem}

The following essentially comes from \cite{Wo66}, but also follows from \cite{ScVi16}.

\begin{theorem} \label{SOsr} If $q$ is the power of an odd prime and $n \geq 3$, then $\SO^{\pm}(n, q)$ is strongly real if and only if it is real, and this group is real if and only if $n \not\equiv 2($mod $4)$.
\end{theorem}

The following is the result \cite[Theorem 8.6]{KnTh98} of Kn\"{u}ppel and Thomsen, specialized to the case of finite fields.

\begin{theorem} \label{Ksr} If $q$ is the power of an odd prime, and $n \geq 3$ (and $q >3$ when $n < 6$), then $\K^{\pm}(n, q)$ is strongly real if and only if:
\begin{enumerate}
\item[(i)] $q \equiv 1(\mmod 4)$, or
\item[(ii)] $q \equiv 3($mod $4)$ and is one of $\K^+(4m+2, q)$, $\K^+(4m+3, q)$, $\K^-(4m, q)$, or $\K^-(4m+1, q)$, or 
\item[(iii)] $n=8$ or $n=9$.
\end{enumerate}
\end{theorem}

Kn\"{u}ppel and Thomsen also prove the following in \cite[Theorem 8.5]{KnTh98} (see also \cite{Ga10}).

\begin{theorem} \label{Omsr} If $q$ is the power of an odd prime, and $n \geq 3$ (and $q >3$ when $n < 6$), then $\Om^{\pm}(n, q)$ is strongly real if and only if:
\begin{enumerate}
\item[(i)] $q \equiv 1($mod $4)$ and $n \not\equiv 2($mod $4)$, or
\item[(ii)] $q \equiv 3($mod $4)$ and it is $\Om^-(4m, q)$, or 
\item[(iii)] $n=8$ or $n=9$.
\end{enumerate}
\end{theorem}

Note that for $n < 6$ we have $\Om^{\pm}(n,q)$ is isomorphic to other well-understood finite groups, by \cite[Proposition 2.9.1]{KL90}.

Recall that if $F$ is a field with fixed algebraic closure $\overline{F}$, given a monic non-constant polynomial with nonzero constant $f(t) \in F[t]$, the \emph{reciprocal polynomial} of $f(t)$, denoted $f^*(t) \in F[t]$, is the polynomial with the property that for any zero $\alpha \in \overline{F}^{\times}$ of $f(t)$, $\alpha^{-1}$ is a zero of $f^*(t)$ with the same multiplicity.  Then $f(t)$ is \emph{self-reciprocal} if $f(t) = f^*(t)$.

We next describe an important orthogonal decomposition of the underlying space $V$.  Given any $g \in \gO^{\pm}(n, q)$, then as in \cite[Section 3]{KnTh98}, $V$ may be decomposed as an orthogonal (with respect to $B$) direct sum, $V = \oplus V_i$, such that each $V_i$ is an orthogonally indecomposable $g$-invariant subspace of $V$, where each $V_i$ is of one of the following forms:
\begin{enumerate}
\item[(1)] $V_i$ is a (non-orthogonal) direct sum, $V_i = U_i \oplus W_i$ of degenerate totally isotropic $g$-invariant cyclic subspaces (with $\dim_{\F_q}(U_i) = \dim_{\F_q}(W_i)$), and $g$ has a single elementary divisor on both $U_i$ and $W_i$ which is either $(t-1)^{2e}$ (\emph{type $1^-$}) or $(t+1)^{2e}$ (\emph{type $1^+$});
\item[(2)] $V_i$ is $g$-cyclic, such that $g$ has a single elementary divisor $f(t)^e$ on $V_i$, where $f(t)$ is irreducible and self-reciprocal, and either $f(t) \neq t \pm 1$ (\emph{type $2^*$}), or $e$ is odd and $f(t) = t- 1$ (\emph{type $2^-$}) or $f(t) = t+1$ (\emph{type $2^+$});
\item[(3)] $V_i$ is $g$-cyclic, and $V_i$ is a (non-orthogonal) direct sum $V_i = U_i \oplus W_i$ of totally isotropic $g$-invariant subspaces (with $\dim_{\F_q}(U_i) = \dim_{\F_q}(W_i)$), such that $g$ has a single elementary divisor $f(t)^e$ on $U_i$, and a single elementary divisor $f^*(t)^e$ on $W_i$, such that $f(t)$ is irreducible and $f(t) \neq f^*(t)$ (\emph{type $3$}).
\end{enumerate}
In case (1), we say that $V_i$ is \emph{bicyclic} with respect to $g$.

\subsection{Reality results}

We begin with the following observation.
\begin{proposition} \label{easycases}  Let $G$ be one of the finite simple groups $G = \Omega^-(4m+2,q)$ with $q \equiv 1(\mmod 4)$ or $G = \Omega^+(4m+2, q)$ with $q \equiv 3(\mmod 4)$, where $m \geq 1$.  Then all real classes of $G$ are strongly real.
\end{proposition}
\begin{proof} We have $G$ is a subgroup of $\SO^\pm(4m+2, q)$, and all real classes of this group are strongly real by Theorem \ref{SOsr}.  Let $g \in G$ be real in $G$, which is then real in $\SO^{\pm}(4m+2, q)$.  So we may write $g = s_1s_2$, where $s_1^2 = s_2^2 = 1$, and we may assume $s_1, s_2 \in \SO^{\pm}(4m+2,q) \setminus G$.  We have $-I \not\in G$, and so the involutions $-s_1, -s_2 \in G$ with $g = (-s_1)(-s_2)$.  Thus $g$ is strongly real in $G$.
\end{proof}

The goal for this section is to prove that the same statement holds true in the case $G = \Omega^+(4m+2,q)$ with $q \equiv 1($mod $4)$ and for the simple group $G/Z = \POm^+(4m+2,q)$ with $q \equiv 1($mod $4)$, but in contrast we also show that the groups $\Omega^-(4m+2,q)$ and $\POm^-(4m+2,q)$ with $q \equiv 3($mod $4)$ always have weakly real classes.

\begin{lemma} \label{CasesLemma} Let  $g \in \gT^{\pm}(n, q)$, with $V = \F_q^n$.  Let $V = \oplus V_i$ be the orthogonal decomposition corresponding to $g$.  Then $g$ is strongly real in $\gT^{\pm}(n,q)$ if:
\begin{enumerate}
\item there exists a subspace $V_i$ such that $\dim(V_i) \equiv 2(\mmod 4)$ which is type $2^*$ or $3$, or
\item there exists a subspace $V_i$ which is type $2^{\pm}$ and $dV_i = (\F_q^{\times})^2$.
\end{enumerate}
\end{lemma}
\begin{proof} First, since $g \in \gO^{\pm}(n, q)$, then $g$ is the product of two involutions in $\gO^{\pm}(n,q)$ by Theorem \ref{Osr}.  Let us take $g=s h$ with $s^2 = h^2 = 1$, and we assume $s, h \in \gO^{\pm}(n,q) \setminus \gT^{\pm}(n,q)$, otherwise $g$ is already strongly real in $\gT^{\pm}(n,q)$.  Write $s_i = s_{V_i}$ for each subspace $V_i$, and by the proof of \cite[Proposition 4.5]{KnTh98}, we may assume $s_i \in \gO(V_i)$.  First suppose that there is a subspace $V_i$ with $\dim(V_i) \equiv 2($mod $4)$ which is of type $2^*$ or type $3$.  We consider the spinor norm on $\gO(V_i)$, and by \cite[Lemma 5.8]{KnTh98} and its proof, for any non-square $\alpha \in \F_q^{\times} \setminus (\F_q^{\times})^2$, we may replace $s_i$ with some involution $s_i'$ such that $\theta(s_i') = \alpha \theta(s_i)$, where $s_i'$ is a $\mathrm{GL}(V_i)$-conjugate of $s_i$ so that $\det(s_i) = \det(s_i')$.  Then we define
$$ s' = s_i' \oplus \bigoplus_{j \neq i} s_j,$$
and by construction $s' \in \gT^{\pm}(n,q)$, and $s'$ is an inverting involution for $g$.

Next assume that there is a subspace $V_i$ which is type $2^{\pm}$ and has discriminant $dV_i = (\F_q^{\times})^2$.  Then we define
$$s' = (-s_i) \oplus \bigoplus_{j \neq i} s_j,$$
so that $s'$ is an involution in $\gO^{\pm}(n,q)$.  Since $dV_i = (\F_q^{\times})^2$, then $-I_{V_i} \in \mathrm{ker}(\theta_i)$ (by \cite[pgs. 221-222, 230]{CaEn04}), and so $\theta(s') = \theta(s)$.  Since $V_i$ is type $2^{\pm}$, then $\dim(V_i)$ is odd, and so $\det(s') = -\det(s)$.  So by construction we now have $s' \in \gT^{\pm}(n,q)$ and $s'$ is an inverting involution for $g$.
\end{proof}

\begin{proposition} \label{T1mod4}  Let $n \geq 3$ and $q \equiv 1(\mmod 4)$.  Then $\gT^{\pm}(n,q)$ is strongly real.
\end{proposition}
\begin{proof} Let $g \in \gT^{\pm}(n,q)$ with $V = \F_q^n = \oplus V_i$ the orthogonal decomposition corresponding to $g$.  By Lemma \ref{CasesLemma}, we may assume that any $V_i$ of type $2^*$ or $3$ has dimension divisible by 4, and any $V_i$ of type $2^{\pm}$ has discriminant $dV_i \neq (\F_q^{\times})^2$.

For any subspace $V_i$ of type $1^{\pm}$ or $2^-$, it follows from \cite[Lemma 4.7(c)]{KnTh98} that $g_i = g_{V_i} \in \Omega(V_i)$.  Since $q \equiv 1($mod $4)$ and $\mathrm{dim}(V_i)$ must be either odd or divisible by 4 in these cases, it follows from Theorem \ref{Omsr} that $g_i $ is strongly real in $\Omega(V_i)$.

Consider a subspace $V_i$ of type $2^+$.  Then $\mathrm{det}(g_i) = -1$ by definition, and by assumption $dV_i \neq (\F_q^{\times})^2$.  From \cite[Lemma 4.7(d)]{KnTh98}, we also have $\theta(g_i) = dV_i \neq (\F_q^{\times})^2$, and it follows that we have $g_{i} \in \gT(V_i) \setminus \Omega(V_i)$.   By \cite[Lemma 5.1]{KnTh98}, $g_{i}$ can be inverted by an orthogonal involution $s_{i}$ such that either $\mathrm{det}(s_{i}) = 1$ and $\theta(s_{i}) = \pm (\F_q^{\times})^2 = (\F_q^{\times})^2$ (since $q \equiv 1($mod $4)$), or such that $\det(s_{i}) = -1$ and $\theta(s) = \pm dV_i \neq (\F_q^{\times})^2$ (by our assumption on $dV_i$ and since $q \equiv 1($mod $4)$).  In either case, we have $s_{i} \in \gT(V_i)$.  

 Now let $Y$ be the orthogonal direct sum of all subspaces of type $2^*$ or type $3$, and so $Y$ has dimension divisible by $4$, and let $X$ be the orthogonal direct sum of all other subspaces $V_i$.  If $g \in \Omega^{\pm}(n,q)$, then the number of subspaces of type $2^+$ must be even, and it follows from the previous two paragraphs that $g_X \in \Omega(X)$.  Then we must also have $g_Y \in \Omega(Y)$.  If $g \in \gT^{\pm}(n,q) \setminus \Omega^{\pm}(n,q)$, then $g \not\in \SO^{\pm}(n,q)$ and so the number of subspaces of type $2^-$ must be odd.  Again, from the above it then follows that $g_X \in \gT(X) \setminus \Omega(X)$, and so we must have $g_Y \in \Omega(Y)$.  In any case, we have $g_Y \in \Omega(Y)$ with $\mathrm{dim}(Y)$ divisible by $4$ and $q \equiv 1($mod $4)$, and so $g_Y$ is strongly real in $\Omega(Y)$ by Theorem \ref{Omsr}.  We now have $g = g_X \oplus g_Y$ with $g_X$ strongly real in $\gT(X)$ and $g_Y$ strongly real in $\Omega(Y)$, so that $g$ is strongly real in $\gT^{\pm}(n,q)$.
\end{proof}

\begin{theorem} \label{realomega} Let $q \equiv 1(\mmod 4)$ and $m \geq 1$.  In either of the groups $G=\Omega^+(4m+2,q)$ or $G/Z = \POm^+(4m+2,q)$, every real class is strongly real.
\end{theorem}
\begin{proof}  From Theorem \ref{Ksr} and Proposition \ref{T1mod4}, we have both $\K^+(4m+2,q)$ and $\gT^+(4m+2,q)$ are strongly real.  Letting $g \in \Omega^+(4m+2,q)$, we have $g$ can be inverted by an involution $\sigma \in \K^+(4m+2,q)$ and by an involution $s \in \gT^+(4m+2,q)$.  We may assume $\sigma, s \not\in \Omega^+(4m+2,q)$.

Let $V = \F_q^{4m+2}$ and $V = \oplus V_i$ be the orthogonal decomposition of $V$ corresponding to $g$.  Suppose that $V$ has a subspace $V_j$ of type $2^{\pm}$.  Then $V_j$ has odd dimension and $\det(-I_{V_j}) = -1$.  If $dV_j = (\F_q^{\times})^2$, then by \cite[pgs. 230, 222]{CaEn04} we have $\theta(-I_{V_j}) = (\F_q^{\times})^2$.  If we write $\sigma = \oplus \sigma_i$ with $\sigma_i = \sigma_{V_i}$, then define $\sigma' = (-\sigma_j) \oplus \bigoplus_{i \neq j} \sigma_i$.  Then we have $\det(\sigma') = -\det(\sigma)$ and $\theta(\sigma') = \theta(\sigma)$, so that $\sigma' \in \Omega^+(4m+2,q)$ and $\sigma'$ is an involution which inverts $g$, so $g$ is strongly real in $\Omega^+(4m+2,q)$.  If $dV_j \neq (\F_q^{\times})^2$, then $\theta(-I_{V_j}) \neq (\F_q^{\times})^2$, again from \cite[pgs. 230, 222]{CaEn04}.  Writing $s = \oplus s_i$, we define $s' = (-s_j) \oplus \bigoplus_{i \neq j} s_i$.  Now we have $\det(s') = -\det(s)$ and $\theta(s') \neq \theta(s)$, so that $s' \in \Omega^+(4m+2,q)$ is an inverting involution for $g$.  Thus if $V$ has a subspace of type $2^{\pm}$, then $g$ is strongly real in $\Omega^+(4m+2,q)$.

Now suppose $g \in \Omega^+(4m+2,q)$ is a real element in this group.  In particular, $g$ is then real as an element of $\SO^+(4m+2,q)$.  By \cite[Theorem 7.2]{ScVi16}, this implies that $g$ has an elementary divisor of the form $(t \pm 1)^e$ with $e$ odd, which means that $V$ has a subspace of type $2^{\pm}$.  Thus $g$ is strongly real in $\Omega^+(4m+2,q)$, concluding the proof of the first claim.

To prove the statement in the group $G/Z = P\Omega^+(4m+2,q)$, where $Z = \{ \pm I \}$, we must also consider elements $g \in \Omega^+(4m+2,q) = G$ such that $g$ is conjugate to $-g^{-1}$ in $G$, but $g$ is not real in $G$ (otherwise we already know $g$ is strongly real in $G$).  Since $g$ and $-g^{-1}$ are conjugate, then for any elementary divisor $f(t)^e$ of $g$, there is also an elementary divisor $\tilde{f}(t)^e$ of $g$ with the same multiplicity, where if the roots of $f(t)$ in an algebraic closure $\overline{\F}_q$ are given by the Frobenius orbit of $\alpha \in \overline{\F}_q^{\times}$, then the roots of $\tilde{f}(t)$ are given by the Frobenius orbit of $-\alpha^{-1}$ (noting that the action $\alpha \mapsto \alpha^q$ commutes with $\alpha \mapsto -\alpha^{-1}$).

Continue to write $V = \oplus_i V_i$ for the orthogonal decomposition of $V$ corresponding to $g$.  Since we are assuming $g$ is not real in $G$, then from the above we may assume $V$ has no subspaces of type $2^{\pm}$.  Note also that subspaces of $V$ which are type $1^{\pm}$ have dimension divisible by 4.  We now consider subspaces $V_i$ of type $2^*$ or type $3$.

Consider a subspace $V_i$ of type $2^*$, which then corresponds to an elementary divisor $f(t)^e$ on $V_i$, where $f(t)$ is irreducible in $\F_q[t]$ and is self-reciprocal, with $f(t) \neq t \pm 1$.  That is, the set of roots of $f(t)$ in $\overline{\F}_q$ is invariant under the map $\alpha \mapsto \alpha^{-1}$.  In particular, $f(t)$ has even degree.  Then also $\tilde{f}(t)^e$ is an elementary divisor, where $\tilde{f}(t)$ is also self-reciprocal, corresponding to a subspace $\tilde{V}_i$ of type $2^*$ with the same dimension as $V_i$.  If $f(t) \neq \tilde{f}(t)$, then $V_i \oplus \tilde{V}_i$ has dimension divisible by $4$.  If $f(t) = \tilde{f}(t)$, then the set of roots of $f(t)$ is invariant under both the maps $\alpha \mapsto \alpha^{-1}$ and $\alpha \mapsto -\alpha^{-1}$, and so is also invariant under $\alpha \mapsto -\alpha$.  This implies the only nonzero coefficients of $f(t)$ are those of even powers of $t$, so that $f(t) = f_1(t^2)$ for some irreducible self-reciprocal polynomial $f_1(t) \in \F_q[t]$.  Also $f_1(t) \neq t \pm 1$ since neither of $t^2 \pm 1$ are irreducible (since $q \equiv 1($mod $4)$), and so $f_1(t)$ must have even degree.  Then $f(t) = f_1(t^2)$ has degree divisible by $4$, and it follows that the direct sum of all subspaces of $V$ of type $2^*$ has dimension divisible by $4$.

Finally, consider a subspace $V_i$ of type $3$, so $V_i$ is a non-orthogonal direct sum, $V_i = U_i \oplus W_i$, and $g$ has elementary divisors $f(t)^e$ on $U_i$ and $f^*(t)^e$ on $W_i$.  If $\tilde{f}(t) \neq f(t), f^*(t)$, then $V$ also has a subspace $\tilde{V}_i = \tilde{U}_i \oplus \tilde{W}_i$ with elementary divisors $\tilde{f}(t)^e$ and $\tilde{f}^*(t)^e$.  In this case, $V_i$ and $\tilde{V}_i$ each have the same even dimension, and $V_i \oplus \tilde{V}_i$ has dimension divisible by $4$.  If either $f(t) = \tilde{f}(t)$ or $f^*(t) = \tilde{f}(t)$, then the set of roots of $f(t)$ is invariant under one of the order 2 actions $\alpha \mapsto -\alpha$ or $\alpha \mapsto -\alpha^{-1}$, and so $f(t)$ must have even degree.  Then $U_i$ and $W_i$ each have even degree, and so $V_i$ has dimension divisible by $4$.

Through exhausting all cases, it follows that if $g$ is not real but $g$ is conjugate to $-g^{-1}$ in $G$, then $\dim(V)$ is divisible by $4$.  This contradicts the fact that $\dim(V) =4m+2$, and so $G$ has no such elements.  Thus all real classes of $G/Z = \POm^+(4m+2, q), \, q \equiv 1(\mmod 4)$, are strongly real.
\end{proof}

We now assume $q \equiv 3($mod $4)$ and $m \geq 1$, with $G = \Om^-(4m+2,q)$ and $G/Z = \POm^-(4m+2,q)$, and we show both $G$ and $G/Z$ always have \emph{weakly real} elements, that is, elements which are real but not strongly real.

\begin{theorem} \label{Weakly}
Let $q \equiv 3(\mmod 4)$ and $m \geq 1$.  Then $\POm^-(4m+2,q)$ contains weakly real elements.
\end{theorem}
\begin{proof} In Lemma \ref{Om6wr} of the Appendix, we construct an element $h \in \Om^-(6,q)$ such that $hZ$ is weakly real in $\POm^-(6,q)$, where the elementary divisors of $h$ are $(t-1)^2, (t-1)^2, t+1, t+1$.  Now consider an element $\eta \in \Om^+(8,q)$ such that $\eta$ has elementary divisors $(t^2 + 1)^2, (t^2 + 1)^2$, which exists by \cite[Propositions 16.10, 16.30]{CaEn04}, and we may define an element $g_1 \in \Om^-(8l+6,q)$ as a block-diagonal direct sum as
\begin{equation} \label{wrbuild}
g_1 = h \oplus \bigoplus_{i = 1}^l \eta,
\end{equation}
and we show that $g_1 Z$ is weakly real in $\POm^-(8l+6,q)$.  First, we have $\eta_l = \bigoplus_{i=1}^l \eta \in \Om^+(8l,q)$, and since $\Om^+(8,q)$ is strongly real by Theorem \ref{Omsr}, then $\eta_l$ is strongly real in $\Om^+(8l,q)$.  It follows from \cite[Proposition 16.34]{CaEn04} that the centralizer of $\eta_l$ in $\gO^+(8l,q)$ is contained completely in $\Om^+(8l,q)$.  Thus any element in $\gO^+(8l,q)$ which conjugates $\eta_l$ to its inverse must be in $\Om^+(8l,q)$.  By \cite[Case (C), (iv) of 2.6]{Wall}, since the eigenvalues of $h$ are distinct from those of $\eta_l$, we have
$$ C_{\gO^-(8l+6,q)}(g_1) \cong C_{\gO^-(6,q)}(h) \times C_{\gO^+(8l,q)}(\eta_l) = C_{\gO^-(6,q)}(h) \times C_{\Om^+(8l,q)}(\eta_l).$$
We therefore have
\begin{equation} \label{wrcent}
C_{\Om^-(8l+6,q)}(g_1) \cong C_{\Om^-(6,q)}(h) \times C_{\Om^+(8l,q)}(\eta_l).
\end{equation}
Next notice that $g_1$ is real in $\Om^-(8l+6,q)$ since $h$ is real in $\Om^-(6,q)$ and $\eta_l$ is real in $\Om^+(8l,q)$, and so $g_1Z$ is real in $\POm^-(8l+6,q)$.  In particular, there is an element $ x_1 = (x, x_l) \in \Om^-(6,q) \times \Om^+(8l,q)$ such that $xhx^{-1} = h^{-1}$ and $x_l \eta_l x_l^{-1} = \eta_l^{-1}$, and so $x_1 g_1 x_1^{-1} = g_1^{-1}$ in $\Omega^-(8l+6,q)$.  Since any two elements which conjugate $g_1$ to its inverse differ by an element of the centralizer, then from \eqref{wrcent} we have every element of $\Omega^-(8l+6,q)$ which conjugates $g_1$ to its inverse must be of the form $(x, x_l) \in \Om^-(6,q) \times \Om^+(8l,q)$.  If $g_1 Z$ is strongly real in $\POm^-(8l+6,q)$, and since the elementary divisors of $g_1$ prevent the possibility that $g_1$ is conjugate to $-g_1^{-1}$, we suppose that $x_1 g_1 x_1^{-1} = g_1^{-1}$ with $x_1^2 = \pm I$.  Then $x_1 = (x, x_l)$ with $x^2 = \pm I$ and $xhx^{-1} = h^{-1}$ for $x \in \Om^-(6,q)$.  This contradicts the fact that $hZ$ is weakly real in $\Om^-(6,q)$, and thus $g_1 Z$ must be weakly real in $\POm^-(8l+6, q)$.

In Lemma \ref{Om10wr}, we construct an element $h_0 \in \Omega^-(10,q)$ such that $h_0 Z$ is weakly real in $\POm^-(10,q)$, where $h_0$ has elementary divisors $$(t-1)^3, (t-1)^3, (t+1)^2, (t+1)^2.$$  Then we consider the element $g_0 = (h_0, \eta_l) \in \Om^-(10, q) \times \Om^+(8l,q)$, so we repeat the construction in \eqref{wrbuild} with $h$ replaced by $h_0$.  By repeating the same argument, we obtain $g_0 Z$ is weakly real in $\POm^-(8l+10,q)$.  This gives that every $\POm^-(4m+2, q)$, $q \equiv 3(\mmod 4)$, $m \geq 1$, has weakly real elements.
\end{proof}

\noindent {\bf Remark.} In the above construction, it would of course be much simpler to just add a single elementary divisor of the form $(t^2 + 1)^2$, so that only one base case is needed.  However, this does not work because such an element is not real in $\Omega^+(4,q)$ when $q \equiv 3($mod $4)$.

\section{Characteristic two} \label{CharTwo}

We now assume that $q$ is a power of 2, and consider the vector space $V = \F_q^{2n}$ with $Q$ a non-defective quadratic form on $V$, and $B$ the associated non-degenerate symplectic form on $V$.  All non-degenerate symplectic forms on $V$ are equivalent, and so we may always embed the orthogonal group $\gO^{\pm}(2n, q)$ associated with $Q$ in the symplectic group $\Sp(2n, q)$ associated with $B$, and there is only one such symplectic group up to isomorphism.  We also recall that given $g \in \gO^{\pm}(2n, q)$, we have $g \in \Om^{\pm}(2n, q)$ if and only if $\mathrm{rank}(g + I)$ is even \cite[Proposition 3.2]{Ra11}.

Given any $g \in \Sp(2n, q)$ (and so any $g \in \Om^{\pm}(2n, q)$), we now describe an orthogonal decomposition (with respect to $B$) of $V = \F_q^{2n}$ which is somewhat similar to the case when $q$ is odd.  We follow the description given in \cite[Section 1]{Gow81}, and these results are due to Huppert \cite{Hu80}.  We say a subspace $W$ of $V$ is symplectically indecomposable with respect to $g$ if $W$ has no $B$-orthogonal decomposition into non-trivial subspaces which are $g$-invariant.  Then $V$ can be orthogonally decomposed, as $V = \oplus V_i$, into symplectically indecomposable $g$-invariant subspaces $V_i$ which are one of the following forms:
\begin{enumerate}
\item[(1)] $V_i$ is a (non-orthogonal) direct sum, $V_i = U_i \oplus W_i$, of degenerate totally isotropic $g$-invariant cyclic subspaces (with $\dim_{\F_q}(U_i) = \dim_{\F_q}(W_i)$), and $g$ has a single elementary divisor $(t-1)^e$ on both $U_i$ and $W_i$;
\item[(2)] $V_i$ is $g$-cyclic, such that $g$ has a single elementary divisor $f(t)^e$, where $f(t)$ is irreducible and self-reciprocal, and if $f(t) = t-1$ then $e$ is even;
\item[(3)] $V_i$ is $g$-cyclic, and $V_i$ a (non-orthogonal) direct sum $V_i = U_i \oplus W_i$ of totally isotropic $g$-invariant subspaces (with $\dim_{\F_q}(U_i) = \dim_{\F_q}(W_i)$), such that $g$ has a single elementary divisor $f(t)^e$ on $U_i$, and a single elementary divisor $f^*(t)^{e}$ on $W_i$, such that $f(t)$ is irreducible and $f(t) \neq f^*(t)$.
\end{enumerate}
We continue to say in case (1) that $V_i$ is \emph{bicyclic} with respect to $g$.  Note that $e$ can be even or odd in case (1), while $e$ must be even in case (2) (different from the case when $q$ is odd).  Given the decomposition $V = \oplus V_i$, we will also continue to write $\gO(V_i)$ and $\Om(V_i)$ for the groups corresponding to the quadratic form $Q$ restricted to the subspace $V_i$.

R\"{a}m\"{o} proved that every element of $\Om^{\pm}(4m, q)$ is strongly real by considering the decomposition $V = \oplus V_i$.  By applying the work done there, we obtain the following statement for the group $\Om^{\pm}(4m+2, q)$.

\begin{proposition} \label{CharTwoTwist}  Let $q$ be a power of $2$, and let $g \in \Om^{\pm}(4m+2, q)$.  Then there exists an element $h \in \gO^{\pm}(4m+2, q) \setminus \Om^{\pm}(4m+2, q)$ such that $h^2 = 1$ and $h g h^{-1} = g^{-1}$.
\end{proposition}
\begin{proof}  Write $V = \oplus V_i$, where each $V_i$ is either cyclic or bicyclic, $\dim(V) = 4m+2$, and $g = \oplus g_i$, where $g_i = g_{V_i}$.  Note that we must have each $\dim(V_i)$ even, since otherwise $V_i$ is cyclic of odd dimension which is not possible.  We consider each possibility for $V_i$.

First suppose $4|\dim(V_i)$.  If $V_i$ is cyclic, then by \cite[Proposition 3.3]{Ra11} there is an involution $h_i \in \Om(V_i)$ which inverts $g_i$.  If $V_i$ is bicyclic, then by \cite[Proposition 3.5]{Ra11} there is an involution $h_i \in \Om(V_i)$ which inverts $g_i$.

Now suppose $\dim(V_i) \equiv 2($mod $4)$.  If $V_i$ is cyclic, then by \cite[Proposition 3.3]{Ra11}, there is an involution $h_i \in \gO(V_i) \setminus \Om(V_i)$ which inverts $g_i$.  If $V_i$ is bicyclic, then by \cite[Propositions 3.16 and 3.17]{Ra11}, there exists an involution $h_i \in \Om(V_i)$ and an involution $h_i' \in \gO(V_i) \setminus \Om(V_i)$, each of which inverts $g_i$.

Since $\dim(V)=4m+2$, then there are an odd number of $V_i$'s such that $\dim(V_i) \equiv 2($mod $4)$.  So, when taking $h = \oplus h_i$, we may choose an odd number of the $h_i$'s such that $\mathrm{rank}(h_i + I_{V_i})$ is odd.  Thus $\mathrm{rank}(h + I)$ is odd, and so $h$ inverts $g$,  $h^2 = 1$, and $h \in \gO^{\pm}(4m+2, q) \setminus \Om^{\pm}(4m+2, q)$.
\end{proof}

We will also need the following statement.

\begin{lemma} \label{RamoExt} If $q$ is a power of $2$ and $u \in \Om^{\pm}(4m,q)$ has a single elementary divisor corresponding to a cyclic space, then $u$ has an inverting involution $h \in \gO^{\pm}(4m,q) \setminus \Om^{\pm}(4m,q)$.
\end{lemma}
\begin{proof} This follows directly from the proof of \cite[Proposition 3.4]{Ra11}, where we get $a_{n/2} = a_{2m} = \binom{4m}{2m} = 0$ in characteristic 2.  The rank of the resulting inverting involution is then odd, and so is in $\gO^{\pm}(4m, q) \setminus \Om^{\pm}(4m, q)$. \end{proof}

R\"{a}m\"{o} also proves that for a unipotent cyclic subspace $V_i$, if $\dim(V_i) \equiv 2($mod $4)$, then there is an inverting involution $h_i' \in \Om(V_i)$ \cite[Proposition 3.4]{Ra11}.  This means that the only case where there may not be an inverting involution in $\Om(V_i)$ is in the case $V_i$ is cyclic, non-unipotent, $\dim(V_i) \equiv 2($mod $4)$, and $V_i$ corresponds to a self-dual elementary divisor $f(t)^e$.  We must have $\deg(f(t))$ even since $f(t)$ is self-dual, and so $e$ must be odd.  We are concerned with the case that these have odd multiplicity, that is, when there are an odd number of such elementary divisors counting multiplicity (otherwise, we have an even number of $h_i \in \gO(V_i) \setminus \Om(V_i)$, and we will get $h = \oplus h_i \in \Om^{\pm}(n,q)$).  These observations motivate the main result of this section, which we now prove.

\begin{theorem} \label{OmStrong}  Let $q$ be a power of $2$, $V = \F_q^{4m+2}$, and $g \in \Om^{\pm}(4m+2,q)$.  Then $g$ is strongly real in $\Om^{\pm}(4m+2,q)$ if and only if either of the following holds:
\begin{enumerate}
\item[(i)] There are an even number of elementary divisors (counting multiplicity) of $g$ of the form $f(t)^{e}$ with $e$ odd, $f(t)$ self-reciprocal, and $\deg(f(t)) = 4r+2$, or
\item[(ii)] In the orthogonal decomposition $V = \oplus V_i$, there exists a $V_i$ on which $g$ has only eigenvalues equal to $1$, other than bicyclic $V_i$ with dimension divisible by $4$.
\end{enumerate}
Moreover, all real classes in $\Om^{\pm}(4m+2,q)$ are strongly real.
\end{theorem}
\begin{proof}  We first suppose (i) or (ii) holds.  Write $V = \oplus V_i$ as before, and consider the cases for $V_i$ as in the proof of Proposition \ref{CharTwoTwist}.  If (i) is true, then since $\dim(V) = 4m+2$, there is some other $V_i$ in the orthogonal decomposition such that $\dim(V_i) \equiv 2($mod $4)$.  Such a space is either bicyclic or cyclic and unipotent, in which case we can choose an inverting involution in the correct coset.  If (ii) is true, then we can again choose an inverting involution in the correct coset, by Lemma \ref{RamoExt} and \cite[Propositions 3.4, 3.16, and 3.17]{Ra11}.

For the converse, we assume $g$ does not satisfy (i) or (ii), and we prove that $g$ is not real in $\Om^{\pm}(4m+2,q)$, and this also proves the second claim.  First, by Proposition \ref{CharTwoTwist}, there is an inverting involution $h \in \gO^{\pm}(4m+2,q) \setminus \Om^{\pm}(4m+2, q)$.  If $k \in \gO^{\pm}(4m+2, q)$ is any other inverting involution for $h$, then we must have $k = ha$ where $a \in C_{G}(g)$ and $G = \gO^{\pm}(4m+2, q)$.  Now write $V = V_y \oplus V_u$, where $V_y$ is the direct sum of all $V_i$ on which $g$ has eigenvalues different from $1$, and $V_u$ is the direct sum of those on which $g$ has eigenvalues only equal to $1$.  Write $g = g_y \oplus g_u$, where $g_y = g_{V_y}$ and $g_u = g_{V_u}$.  From the description of centralizers given by Wall \cite[Section 3.7]{Wall}, it follows that we have the direct product
\begin{equation} \label{cent}
C_G(g) = C_{\gO(V_y)}(g_y) \times C_{\gO(V_u)}(g_u).
\end{equation}
Since $g$ does not satisfy condition (ii), then neither does $g_u$.  This implies that $g_u$ is an \emph{exceptional} unipotent element of $\gO(V_u)$, as defined in \cite[pg. 2547]{FuSaTi12}.  It then follows from \cite[Theorem 2.5(ii)(a)]{FuSaTi12} that $C_{\gO(V_u)}(g_u) \subset \Om(V_u)$, and so the conjugacy class of $g_u$ in $\gO(V_u)$ splits into two conjugacy classes in $\Omega(V_u)$.  From \eqref{cent}, the conjugacy class of $g = g_y \oplus g_u$ in $G$ must also split into two conjugacy classes in $\Omega^{\pm}(4m+2,q)$, and so
$$ C_G(g)  \subset \Om^{\pm}(4m+2, q).$$
Since $h \in G \setminus \Om^{\pm}(4m+2, q),$ and $C_G(g) \subset \Om^{\pm}(4m+2,q)$, then every inverting element for $g$ in $\Om^{\pm}(4m+2,q)$ must be in the outer coset $h \Om^{\pm}(4m+2,q)$.  Thus $g$ is not real in $\Om^{\pm}(4m+2,q)$, giving the claim.
\end{proof}

\section{Real-valued characters} \label{Indicators}
In this section, we give some results on real-valued irreducible complex characters of finite simple orthogonal groups, and compare them to the known reality results of their conjugacy classes.

Given a finite group $G$, we let $\Irr(G)$ denote the collection of complex irreducible characters of $G$, so that each $\chi \in \Irr(G)$ is afforded by a complex representation $(\pi, W)$ of $G$.  Given a basis $\mathcal{B}$ of $W$, let $[\pi] = [\pi]_{\mathcal{B}}$ denote the associated matrix representation.  We say that $(\pi, W)$ is a {\it real representation} if there is a basis $\mathcal{B}$ of $W$ such that $[\pi(g)]$ has all real entries for all $g \in G$.    The {\it Frobenius-Schur indicator} of $\chi \in \Irr(G)$, which we denote by $\vep(\chi)$, is given by the formula
$$ \vep(\chi) = \frac{1}{|G|} \sum_{g \in G} \chi(g^2), $$
and takes the values $\vep(\chi) = 0$ if $\chi$ is not real-valued, $\vep(\chi) = -1$ if $\chi$ is real-valued but $(\pi, W)$ is not a real representation, and $\vep(\chi) = 1$ if $(\pi, W)$ is a real representation (see \cite[Chapter 4]{Isaacs}).  In particular, because the number of real conjugacy classes of $G$ is equal to the number of real-valued irreducible complex characters of $G$, then $\vep(\chi) = \pm 1$ for all $\chi \in \Irr(G)$ if and only if $G$ is a real group.

There is a twisted variation of the Frobenius-Schur indicator, due to N. Kawanaka and H. Matsuyama \cite{KaMa90}, which we will need here.  Let $\iota$ be an automorphism of the group $G$, such that $\iota^2 = 1$, and let ${^\iota \chi}$ denote the character defined by ${^\iota \chi}(g) = \chi(\iota(g))$.  If we define $\vep_{\iota}(\chi)$ by
$$ \vep_{\iota}(\chi) = \frac{1}{|G|} \sum_{g \in G} \chi(g \cdot \iota(g)),$$
then $\vep_{\iota}(\chi)$ takes only the values $1, -1, 0$ for $\chi \in \Irr(G)$, and $\vep_{\iota}(\chi) = \pm 1$ if and only if ${^\iota \chi} = \overline{\chi}$.  This twisted variant generalizes the Frobenius-Schur indicator, and has many other interesting properties, but we do not need them here.  We have the following, which is only a slight variation of \cite[Lemma 2.3(i)]{TV17}.

\begin{lemma} \label{WeakIndexTwo} Suppose $H$ and $G$ are finite groups with $H$ an index two subgroup of $G$ such that $G = \langle H, s \rangle$ with $s^2 = 1$.  If $\vep(\chi) = 1$ for all $\chi \in \Irr(G)$, then $\vep(\psi) \geq 0$ for all $\psi \in \Irr(H)$.  If $H$ is a real group, then $\vep(\psi) = 1$ for all $\psi \in \Irr(H)$.
\end{lemma}
\begin{proof}  We define $\iota$ on $H$ by ${^\iota g} = s g s^{-1}$.  Given any $\psi \in \Irr(H)$, the induced character $\psi^G$ is either irreducible or the sum of two distinct irreducibles of $G$.  As in the proof of \cite[Lemma 2.3]{TV17}, if $\psi^H = \chi$ irreducible, we have $\vep(\chi) = \vep(\psi) + \epi(\psi)$, while if $\psi^H = \chi_1 + \chi_2$ with $\chi_1, \chi_2 \in \Irr(G)$, we have
$$\vep(\chi_1) + \vep(\chi_2) = \vep(\psi) + \epi(\psi) \; \text{ with } \; \vep(\chi_1) = \vep(\chi_2).$$
Since $\vep(\chi) = 1$ for all $\chi \in \Irr(G)$, and $\vep(\psi)$, $\epi(\psi)$ can only take the values $0$ or $\pm 1$, then in the first case we must either have $\vep(\psi) = 1$ and $\epi(\psi) = 0$, or $\vep(\psi) = 0$ and $\epi(\psi) = 1$.  In the second case, we must have $\vep(\psi) = \epi(\psi) = 1$.  Thus $\vep(\psi) \geq 0$ for all $\psi \in \Irr(H)$ as claimed.  Further, if $H$ is a real group, then $\vep(\psi) = \pm 1$ for all $\psi$, and so we must have $\vep(\psi) = 1$ for all $\psi \in \Irr(H)$.
\end{proof}

We may now obtain the following result on the Frobenius-Schur indicators of the finite simple orthogonal groups.

\begin{theorem} \label{RealVal} Let $q$ be any prime power, $n \geq 1$, and $H = \Om^{\pm}(n, q)$, other than the case $H = \Om^{-}(4m+2, q)$ with $q \equiv 3($mod $4)$.  Then for any $\psi \in \Irr(H)$, we have $\vep(\psi) \geq 0$.  That is, if $\psi$ is a real-valued irreducible character of $H$, then $\psi$ is afforded by a real representation.  The same statement holds for the groups $\POm^{\pm}(n,q)$, other than $\POm^-(4m+2,q)$ with $q \equiv 3($mod $4)$.
\end{theorem}
\begin{proof}  If $q$ is a power of $2$, then this statement is proven in \cite[Theorems 8.3 and 8.6]{Vi202} and in \cite[Theorem 4.2]{Vi20} (noting that $\SO(2m+1,q) \cong \Sp(2m,q)$ when $q$ is a power of $2$).  Thus we now assume that $q$ is the power of an odd prime.

For any $\theta \in \Irr(\gO^{\pm}(n,q))$, we have $\vep(\theta) = 1$ by a result of Gow \cite[Theorem 1]{Gow85}.  For each case,  we consider an index 2 subgroup $G$ of $\gO^{\pm}(n,q)$ such that $G$ is strongly real.  If $n \neq 4m+2$, then we can take $G = \SO^{\pm}(n,q)$ by Threorem \ref{SOsr}.  When $n=4m+2$, then we may take $G = K^{\pm}(4m+2,q)$ by parts (i) and (ii) of Theorem \ref{Ksr}.  In particular, for each of these cases $G$ is a real group and $\gO^{\pm}(n, q) = \langle G, t \rangle$ with $t$ any involution in $\gO^{\pm}(n,q) \setminus G$.  It follows from the second statement in Lemma \ref{WeakIndexTwo} that $\vep(\chi) = 1$ for all $\chi \in \Irr(G)$.  Now $[G: H] = 2$, and $G = \langle H, s \rangle$ with $s$ any involution in $G \setminus H$.  It now follows from the first statement in Lemma \ref{WeakIndexTwo} that $\vep(\psi) \geq 0$ for all $\psi \in \Irr(H)$.   Note that in the case that $H=\Om^{\pm}(n,q)$ is itself a real group, then we have $\vep(\psi) = \pm 1$ for all $\psi \in \Irr(H)$, and so it follows that $\vep(\psi) = 1$ for all $\psi \in \Irr(H)$.  This is the case for the groups given in Theorem \ref{Omsr}.

Now consider the group $\tilde{H} = P\Om^{\pm}(n,q) = H/Z$, where $Z=Z(H)$ is the center of $H$.  Given any $\omega \in \Irr(\tilde{H})$, we may define $\psi$ of $H$ by $\psi(h) = \omega(hZ)$.  It follows from a direct computation that $\langle \psi, \psi \rangle = \langle \omega, \omega \rangle = 1$, and so $\psi \in \Irr(H)$.  Then we may compute
$$ \vep(\omega) = \frac{|Z|}{|H|} \sum_{hZ \in \tilde{H}} \omega(h^2 Z) = \frac{|Z|}{|H|} \sum_{hZ \in \tilde{H}} \psi(h^2) = \frac{1}{|H|} \sum_{h \in H} \psi(h^2) = \vep(\psi).$$
The result now follows, since $\vep(\psi) \geq 0$.
\end{proof}

We cannot make the same conclusion for the group $H= \Om^-(4m+2, q)$ with $q \equiv 3($mod $4)$, where to do so it would be enough to show that $H$ is contained in an index 2 subgroup of $\gO^-(4m+2, q)$ which is real (but not necessarily strongly real).  By Theorem \ref{SOsr}, such an index 2 subgroup cannot be $\SO^-(4m+2,q)$.  In fact, one can show that both of the groups $\K^-(4m+2, q)$ and $\gT^-(4m+2, q)$, with $q \equiv 3($mod $4)$, contain elements which are not real.  That is, in the exceptional case $H = \Om^-(4m+2, q)$ with $q \equiv 3($mod $4)$, $H$ is not contained in any real index two subgroup of $\gO^-(4m+2, q)$.  We could also make this conclusion from the above if we construct some $\chi \in \Irr(\Om^-(4m+2, q))$, $q \equiv 3($mod $4)$, such that $\vep(\chi) = -1$.  By \cite[Proposition 2.9.1(vii)]{KL90}, we have $\POm^-(6,q) \cong \PSU(4,q)$, and the character table for $\PSU(4,3)$ reveals the existence of characters of $\POm^-(6,3)$ which indeed have Frobenius-Schur indicator $-1$ \cite[pg. 54]{Atlas}.

Given the results obtained here, we expect that for any values of $n$ and $q$, all real classes of $\Om^{\pm}(n,q)$ or $\POm^{\pm}(n,q)$ are strongly real if and only if all real-valued irreducible complex characters have Frobenius-Schur indicator equal to 1.  We hope to complete the proof of this statement in the sequel, by constructing irreducible complex characters with Frobenius-Schur indicator $-1$ of $\Om^-(4m+2, q)$ and $\POm^-(4m+2,q)$ with $q \equiv 3($mod $4)$, and showing all real classes are strongly real in the groups $\Om(2n+1,q)$, $\Om^+(4m, q)$, and $\POm^+(4m,q)$ with $q \equiv 3($mod $4)$.

\appendix

\section{Weakly real elements}
Here we construct elements $h \in \Omega^-(6,q)$ and $h_0 \in \Omega^-(10,q)$, with $q \equiv 3(\mmod 4)$, such that $hZ$ is weakly real in $\POm^-(6,q)$ and $h_0 Z$ is weakly real in $\POm^-(10,q)$.

We define the matrix $J_n$ to have 1's on the main antidiagonal and 0's elsewhere, so
$$ J_n = \left( \begin{array}{rrr} &  &  1 \\  &  \adots &  \\ 1 &  & \end{array} \right).$$ 

We first consider an element with elementary divisors $(t-1)^2, (t-1)^2$, given by
$$ u = \left( \begin{array}{rrrr} 1 & -1 & & \\  & 1 & & \\ & & 1 & 1 \\ & & & 1 \end{array} \right).$$
If we define a symmetric form on $\F_q^4$ using the Gram matrix $J_4$, then this is a split form, and by \cite[Proposition 16.30]{CaEn04} we have $ u \in \Omega^+(4,q)$.  It follows from \cite[Proposition 16.34]{CaEn04} that the centralizer of $u$ in $\gO^+(4,q)$ is contained completely in $\Omega^+(4,q)$.  Now note that the element
$$ s_0 = \left( \begin{array}{cc} 0 & I \\ I & 0 \end{array} \right)$$
satisfies $s_0 \in \SO^+(4,q)$ and $s u s^{-1} = u^{-1}$.  The $-1$-eigenspace of $s_0$ is spanned by $(1, 0, -1, 0)^{\top}$ and $(0, 1, 0, -1)^{\top}$, and restricting our symmetric form defined by $J_4$ to this subspace, with these vectors as a basis, gives that the symmetric form on this subspace has matrix $\left( \begin{array}{rr} 0 & -2 \\ -2 & 0 \end{array} \right)$.  Since this form has discriminant a non-square (since $-4$ is not a square when $q \equiv 3(\mmod 4)$), then this is a split form on the $-1$-eigenspace.  By \cite[Proposition 16.30]{CaEn04}, we have $s_0 \not\in \Omega^+(4,q)$.  Since one element in $\SO^+(4,q)$ which conjugates $u$ to its inverse is not in $\Omega^+(4,q)$ and the centralizer of $u$ in $\SO^+(4,q)$ is contained in $\Omega^+(4,q)$, then every element of $\SO^+(4,q)$ which conjugates $u$ to its inverse must be in $\SO^+(4,q) \setminus \Omega^+(4,q)$.

We also claim that $u$ is not conjugated to its inverse by any element of $\gO^+(4,q)$ which squares to $-I$.  By direct computation, such an element would have to be of the form (just from it conjugating $u$ to its inverse):
$$ \left( \begin{array}{rrrr} a & b & c & d \\ 0 & -a & 0 & c \\ c_1 & d_1 & a_1 & b_1 \\ 0 & c_1 & 0 & -a_1 \end{array} \right), $$
where $a^2 + c c_1 = -1$ and $ac + ca_1 = 0$, from the assumption that this squares to $-I$.  We cannot have $c = 0$, since otherwise $a^2 = -1$ with $q \equiv 3(\mmod 4)$, and so we have $a_1 = -a$.  Using the fact that this matrix must preserve the form defined by $J_4$, we obtain $c c_1 - a a_1 = 1$, or $a^2 + c c_1 = 1$.  This is a contradiction to $a^2  + c c_1 = -1$, and so such an element cannot exist.

We now consider the symmetric form defined by the Gram matrix
$$ \left( \begin{array}{rr} J_4 & \\ & I_2 \end{array} \right),$$
which defines a non-split form on $\F_q^6$ when $q \equiv 3(\mmod 4)$, and we consider the element 
$$h = \left( \begin{array}{rr} u & \\  & -I_2 \end{array} \right).$$
Then $h$ has elementary divisors $(t-1)^2, (t-1)^2, t+1, t+1$, and $h \in \Omega^-(6,q)$ by \cite[Proposition 16.30]{CaEn04}.

\begin{lemma} \label{Om6wr}
The element $h \in \Omega^-(6,q)$ is such that $hZ$ is weakly real in $\POm^-(6,q)$ when $q \equiv 3(\mmod 4)$.
\end{lemma}
\begin{proof} First, take any element $x_0 \in \SO^-(2,q) \setminus \Om^-(2,q)$, and take the element $s_0 \in \SO^+(4,q) \setminus \Om^+(4,q)$ defined above.  Then we have $\left( \begin{array}{ll} s_0 & \\  & x_0 \end{array} \right) \in \Om^-(6,q)$, and this element conjugates $h$ to its inverse.  Thus $h$ is real in $\Omega^-(6,q)$ and $hZ$ is real in $\POm^-(6,q)$.

To prove that $hZ$ is not strongly real in $\POm^-(6,q)$, we first note that $h$ cannot be conjugate to $-h^{-1}$ since these two elements have different sets of elementary divisors.  So we consider only the possibility that $s h s^{-1} = h^{-1}$ with $s^2 = \pm I$.  Since $u \in \Om^+(4,q)$ and $-I \in \Om^-(2,q)$ have distinct eigenvalues, then it follows from \cite[Case (C), (iv) of 2.6]{Wall} that
$$ C_{\gO^-(6,q)}(h) = C_{\gO^+(4,q)}(u) \times C_{\gO^-(2,q)}(-I) = C_{\Om^+(4,q)} \times \gO^-(2,q),$$
since the centralizer of $u$ in $\gO^+(4,q)$ is completely contained in $\Om^+(4,q)$ (mentioned above).  That is, we have
$$ C_{\Om^-(6,q)}(h) = C_{\Om^+(4,q)}(u) \times \Om^-(2,q).$$
We now have the element 
$$(s_0, x_0) \in (\SO^+(4,q) \setminus \Om^+(4,q)) \times (\SO^-(2,q) \setminus \Om^-(2,q)),$$ which conjugates $h$ to its inverse, and every such element in $\Om^-(6,q)$ must be a product of this with an element of the centralizer of $h$ in $\Om^-(6,q)$.
So any element $s \in \Om^-(6,q)$ such that $shs^{-1} = h^{-1}$ and $s^2 = \pm I$ must be of the form $s = (\sigma, \tau)$ where $\sigma \in \SO^+(4,q) \setminus \Om^+(4,q)$, $\tau \in \SO^-(2,q) \setminus \Om^-(2,q)$, $\sigma^2 = \pm I$, and $\tau^2 = \pm I$.  As proven above, there is no element of $\gO^+(4,q)$ which conjugates $u$ to its inverse which squares to $-I$, and so we must assume $\sigma^2 = I$ and $\tau^2 = I$.  However, $\SO^-(2,q)$ is cyclic of order $q+1$, and so the only elements $\tau$ in $\SO^-(2,q)$ which square to $I$ are $\pm I$.  But $\pm I$ are both elements of $\Omega^-(2,q)$, contradicting that such an element would have to come from the other coset in $\SO^-(2,q)$.  Thus no such $s \in \Om^-(6,q)$ exists, and we have $hZ$ is weakly real in $\POm^-(6,q)$.
\end{proof}

We next consider the symmetric form on $\F_q^6$ defined by the Gram matrix
$$ \left( \begin{array}{ll} J_3 & \\ & J_3 \end{array} \right),$$
which is a non-split form when $q \equiv 3(\mmod 4)$, and we consider the element
$$ u_1 = \left( \begin{array}{rrrrrr} 1 & -1 & \gamma & & & \\ & 1 & 1 & & & \\ & & 1  & & & \\ & & & 1 & -1 & \gamma \\ & & & & 1 & 1 \\ & & & & & 1 \end{array} \right),$$
where $\gamma \in \F_q$ such that $2\gamma + 1 = 0$ (so if $q$ is a power of $p$, we can take $\gamma = (p-1)/2$).  Then $u_1$ preserves the non-split form above, has elementary divisors $(t-1)^3, (t-1)^3$, and since $u_1$ is unipotent then $u_1 \in \Om^-(6,q)$.  We need the following.

\begin{lemma} \label{interwr}
Any involution in $\SO^-(6,q)$ which conjugates the element $u_1$ to its inverse, must be an element of $\Om^-(6,q)$.
\end{lemma}
\begin{proof}  A somewhat tedious but straightforward calculation reveals that any orthogonal involution which conjugates $u_1$ to its inverse, using the above Gram matrix, must be of the form
$$ \left( \begin{array}{rrrrrr} a & a_1 & a_2 & c & c_1 & c_2 \\  & -a & a_1 &  & -c & c_1 \\  &  & a &  & & c \\ c & c_1 & c_2 & b & b_1 & b_2 \\  & -c & c_1  & &  -b & b_1 \\ & & c & &  &b \end{array} \right),$$
such that $a^2 = b^2$, $a^2 + c^2 = 1$,  and $c(a+b)=0$, among other required relations.  Thus either $c=0$ or $a=-b$.  The determinant of this matrix is $(ab-c^2)^3$, and so for this to be in $\SO^-(6,q)$ we must have $(ab-c^2)^3 = 1$.  If $c \neq 0$, then $a = -b$, which gives
$ (ab-c^2)^3 = (-a^2 - c^2)^3 = -1$.  Thus we must have $c = 0$ and $a^3 b^3 = 1$, and since then $a^2 = b^2 = 1$, we must have $a=b= \pm 1$.  It follows that an involution in $\SO^-(6,q)$ which conjugates $u_1$ to its inverse must be either of the form
$$ \left( \begin{array}{rrrrrr} 1 & a_1 & a_2 &  & c_1 & c_2 \\  & -1 & a_1 &  &  & c_1 \\  &  & 1 &  & &  \\  & c_1 & c_2 & 1 & b_1 & b_2 \\  &  & c_1  & &  -1 & b_1 \\ & &  & &  & 1 \end{array} \right),$$
such that $a_1^2 + c_1^2 + 2a_2 = b_1^2 + c_1^2 + 2b_2 = c_1(a_1 + b_1) + 2c_2 = 0$, or of the form
$$  \left( \begin{array}{rrrrrr} -1 & a_1 & a_2 &  & c_1 & c_2 \\  & 1 & a_1 &  &  & c_1 \\  &  & -1 &  & &  \\  & c_1 & c_2 & -1 & b_1 & b_2 \\  &  & c_1  & &  1 & b_1 \\ & &  & &  & -1 \end{array} \right),$$
such that $a_1^2 + c_1^2 - 2a_2 = b_1^2 + c_1^2 - 2b_2 = c_1(a_1 + b_1) - 2c_2 = 0$.

In order to determine whether these involutions are elements of $\Omega^-(6,q)$, we must calculate whether the restriction of the symmetric form to the $-1$-eigenspace of the involution is split or non-split.  In the first case above, the $-1$-eigenspace is 2-dimensional and spanned by the vectors 
$$(-a_1/2, 1, 0, -c_1/2, 0, 0)^{\top}, (-c_1/2, 0, 0, -b_1/2, 1, 0)^{\top}.$$
In terms of this ordered eigenbasis, the matrix for the symmetric form restricted to this eigenspace is $I_2$, which is a non-split form by \cite[pg. 222]{CaEn04}, and so this involution is an element of $\Om^-(6,q)$ by \cite[Proposition 16.30]{CaEn04}.  In the second case above, the $-1$-eigenspace of the involution is spanned by the vectors $(1, 0, 0, 0, 0, 0)^{\top}$, $(0, -a_1/2, 1, 0, -c_1/2, 0)^{\top}$, $(0, 0, 0, 1, 0, 0)^{\top}$, and $(0, -c_1/2, 0, 0, -b_1/2, 1)^{\top}$.  When restricting our symmetric form to this eigenspace, then with respect to this ordered basis, the restricted symmetric form has Gram matrix
$$ \left( \begin{array}{cccc} 0 & 1 & 0 & 0 \\ 1 & (a_1^2 + c_1^2)/4 & 0 & c_1 (a_1+b_1)/4 \\ 0 & 0 & 0 & 1 \\  0 & c_1 (a_1+b_1)/4 & 1 & (b_1^2 + c_1^2)/4 \end{array} \right),$$
which has determinant 1.  It follows this is a split form by \cite[pg. 222]{CaEn04}, and the involution is again an element of $\Om^-(6,q)$ by \cite[Proposition 16.30]{CaEn04}, giving the result.
\end{proof}

\noindent {\bf Remark. } The centralizer of $u_1$ in $\SO^-(6,q)$ contains elements both in $\Om^-(6,q)$ and its other coset, and so while there are elements outside of $\Om^-(6,q)$ which conjugate $u_1$ to its inverse, there are no such involutions.\\
\\
Finally, we consider the element $ h_0 = \left( \begin{array}{rr} u_1 & \\  & -u \end{array} \right)$, where $u$ is the element of $\Om^+(4,q)$ described above.  Then $h_0$ is an element of $\Om^-(10,q)$ with respect to the symmetric form defined by the Gram matrix
$$ \left( \begin{array}{ccc} J_3 &  & \\  & J_3 & \\ & & J_4 \end{array} \right),$$
and $h_0$ has elementary divisors $(t-1)^3, (t-1)^3, (t+1)^2, (t+1)^2$.

\begin{lemma} \label{Om10wr}
The element $h_0 \in \Om^-(10,q)$ is such that $h_0 Z$ is weakly real in $\POm^-(10,q)$ when $q \equiv 3(\mmod 4)$.
\end{lemma}
\begin{proof}  It was shown above that the element $u$ is conjugated to its inverse by elements in $\SO^+(4,q)$ which are not elements of $\Om^+(4,q)$, and that $u$ is not conjugated to its inverse by any element in $\gO^+(4,q)$ which squares to $-I$.  Thus both of these statements hold for $-u$ as well.

As mentioned above, the centralizer of $u_1$ contains elements in $\SO^-(6,q)$ which are not in $\Om^-(6,q)$, which follows from \cite[Proposition 16.34]{CaEn04}.  Taking any element of $\Om^-(6,q)$ which conjugates $u_1$ to its inverse (for example, any of the involutions described above), and multiplying it by an element of the centralizer of $u_1$ from $\SO^-(6,q) \setminus \Om^-(6,q)$, yields an element $y_1$ from that outer coset which conjugates $u_1$ to its inverse.  We have already found an involution $s_0$ from $\SO^+(4,q) \setminus \Om^+(4,q)$ which conjugates $-u$ to its inverse, and so now the block diagonal element $(y_1, s_0) \in \Om^-(10,q)$ conjugates $h_0 = (u_1, -u)$ to its inverse.  Thus $h_0 Z$ is real in $\POm^-(10,q)$.  Next we show that $h_0 Z$ is not strongly real in $\POm^-(10,q)$, and note that we cannot have $h_0$ conjugate to $-h_0^{-1}$ since these elements have different sets of elementary divisors.  So we must show there is no $s \in \Om^-(10,q)$ such that $s h_0 s^{-1} = h^{-1}$ and $s^2 = \pm I$.

Since $u_1$ and $-u$ have distinct eigenvalues, and since the centralizer of $-u$ in $\gO^+(4,q)$ is completely contained in $\Om^+(4,q)$ by \cite[Proposition 16.34]{CaEn04}, then it follows from \cite[Case (C), (iv) of 2.6]{Wall} that we have
$$ C_{\Om^-(10,q)}(h_0) = C_{\Om^-(6,q)}(u_1) \times C_{\Om^+(4,q)}(-u).$$
Since there is an element  $$(y_1, s_0) \in (\SO^-(6,q) \setminus \Om^-(6,q)) \times (\SO^+(4,q) \setminus \Om^+(4,q)),$$ such that $y_1$ conjugates $u_1$ to its inverse and $s_0$ conjugates $-u$ to its inverse, then from the structure of the centralizer of $h_0$, any element of $\Om^-(10,q)$ which conjugates $h_0$ to its inverse must take this form.  However, there is no element in $\SO^+(4,q) \setminus \Om^+(4,q)$ which conjugates $-u$ to its inverse and which squares to $-I$, so there is no such conjugating element for $h_0$ in $\Om^-(10,q)$.  Since there is no involution in $\SO^-(6,q) \setminus \Om^-(6,q)$ which conjugates $u_1$ to its inverse by Lemma \ref{interwr}, then there can be no involution in $\Om^-(10,q)$ which conjugates $h_0$ to its inverse.  Thus $h_0 Z$ is not strongly real in $\POm^-(10,q)$.
\end{proof}

\end{document}